\documentclass[11pt,reqno]{amsart}  
\usepackage{amsmath,amstext,amsthm,amssymb,amsxtra}
\usepackage[top=1.5in, bottom=1.5in, left=1.25in, right=1.25in]	{geometry}
\usepackage{comment}

\usepackage{color}

\usepackage{mathtools}
\mathtoolsset{showonlyrefs,showmanualtags}

\usepackage{hyperref} 
\hypersetup{
    colorlinks=true,       
    linkcolor=blue,          
    citecolor=magenta,        
    filecolor=magenta,      
    urlcolor=cyan           
}

\usepackage[msc-links]{amsrefs}

\theoremstyle{plain} 
\newtheorem{lemma}[equation]{Lemma} 
\newtheorem{proposition}[equation]{Proposition} 
\newtheorem{theorem}[equation]{Theorem} 
\newtheorem{corollary}[equation]{Corollary} 
\newtheorem{conjecture}[equation]{Conjecture}

\def\Xint#1{\mathchoice
   {\XXint\displaystyle\textstyle{#1}}%
   {\XXint\textstyle\scriptstyle{#1}}%
   {\XXint\scriptstyle\scriptscriptstyle{#1}}%
   {\XXint\scriptscriptstyle\scriptscriptstyle{#1}}%
   \!\int}
\def\XXint#1#2#3{{\setbox0=\hbox{$#1{#2#3}{\int}$}
     \vcenter{\hbox{$#2#3$}}\kern-.5\wd0}}

\def\dashint{\Xint-}
\newcommand{\D}{ \, \mathrm{d}} %
\renewcommand{\roman}[1]{%
  \textup{\uppercase\expandafter{\romannumeral#1}}%
}

\theoremstyle{definition}
\newtheorem{definition}[equation]{Definition} 

\theoremstyle{remark}
\newtheorem{remark}[equation]{Remark}

\allowdisplaybreaks

\numberwithin{equation}{section}

\date{\today}

\title[] {Weighted Weak Type estimates\\ for non-integral Square Functions}

\author{Dario Mena}
\address{Centro de Investigaci\'{o}n en Matem\'{a}tica Pura y Aplicada \\ Escuela de Matem\'{a}tica, Universidad de Costa Rica}
\email {dario.menaarias@ucr.ac.cr}

\author{Maria Carmen Reguera} 
\address{Universidad de M\'alaga, M\'alaga, Spain}
\email {m.reguera@uma.es}

\author{Luz Roncal}
	\address{BCAM -- Basque Center for Applied Mathematics, 48009 Bilbao, Spain \newline
	Universidad del Pa{\'i}s Vasco / Euskal Herriko Unibertsitatea, 48080 Leioa, Spain \newline
	Ikerbasque, Basque Foundation for Science, 48011 Bilbao, Spain}	
    \email{lroncal@bcamath.org}

\thanks{ Darío Mena was partially supported by The University of Costa Rica, through the grant C1012. Maria Carmen Reguera was supported by the Spanish Ministry of Science and Innovation through the projects RYC2020-030121-IAEI/10.13039/501100011033/ and PID2022-136619NB-I00 funded by MCIN/AEI/10.13039/501100011033/FEDER, UE. Luz Roncal was supported by the Basque Government through the BERC 2022-2025 program and by the Spanish Ministry of Science and Innovation: BCAM Severo Ochoa accreditation CEX2021-001142-S/MICIN/AEI/\allowbreak10.13039/501100011033, PID2023-146646NB-I00 funded by MICIU/\allowbreak AEI/10.13039/501100011033 and by ESF+ and CNS2023-143893, and by IKERBASQUE}

\begin{document}
\begin{abstract} 
We provide quantitative weighted weak type estimates for non-integral square functions in the critical case $p=2$ in terms of the $A_p$ and reverse H\"older constants associated to the weight. The method of proof uses a decoupling of the role of the weights via a quantitative version of Gehring's lemma. The results can be extended to other $p$ in the range of boundedness of the square function at hand. 
\end{abstract}

\maketitle

\section{Introduction} 
\label{sec:introduction}

In the present paper, we are interested in quantitative weighted weak $L^p$ estimates for non-integral square functions, ala \cite{AM3}\footnote{By \textit{quantitative weighted estimate} we mean a weighted estimate where the dependence on the characteristic of the weight is explicit. Otherwise, the weighted estimates is \textit{qualitative}.}. The study of sharp weighted estimates for classical square functions, such as the intrinsic square function defined by Wilson in \cites{MR2414491, MR2359017}, has been the source of copious activity in the last decade. Lerner produced the sharp estimate in terms of the constant associated to the Muckenhoupt weight for the strong weighted $L^p$ in \cite{MR2770437}. However, the case of weighted weak $L^p$ estimates is still not completed. While the case $p\neq 2$ is fully understood (see \cites{MR3778152, DLR}), the existing estimate for the critical case $p=2$ is not known to be sharp. Given the intrinsic square function $S_{\operatorname{int}}$ and a weight $w$ in the Muckenhoupt class $A_2$ with constant $ [w]_{A_2}$ (as in \eqref{e.ap} below), the following was conjectured by Domingo-Salazar, Lacey and Rey in \cite{DLR}.
\begin{conjecture} \label{c.weak2}
For $w\in A_2$,
    \begin{align}
\lVert S_{\operatorname{int}}f \rVert_{L^{2,\infty (\omega)}} &\lesssim  [w]_{A_2}^{1/2} \left (1 + \log_{+}[w]_{A_{\infty}} \right )^{1/2} \lVert f \rVert_{L^2 (w)},
\end{align}
and the estimate is sharp.
\end{conjecture}
The most challenging part of this conjecture is to establish the sharpness of the estimate, which is unknown to this date. In contrast to the conjectured estimate, it was established by Ivanisvili, Mozolyako and Volberg  \cite{ivanisvili2022strong}, that the logarithm can be removed from the conjectured bound, if we consider the operator $S_{\operatorname{int}}$ acting on indicator functions.

Motivated by this conjecture, we look at quantitative weighted weak estimates at the critical point $p=2$ for a different class of square functions, namely the non-integral square functions studied in \cite{AM3}; the prototypical models for these operators are the Littlewood--Paley--Stein square functions \cite{St}*{Chapter IV} which are defined as
$$
G_L(f)(x)=\Big(\int_0^{\infty}\big|\nabla e^{-tL}f(x)\big|^2\,dt\Big)^{1/2},\qquad g_L(f)(x)=\Big(\int_0^{\infty}\big|(tL)^{1/2}e^{-tL}f(x)\big|^2\frac{dt}{t}\Big)^{1/2},
$$
where $L$ belongs to some class of elliptic operators. The study of such non-integral square functions has special relevance due to its connection with the Kato square root problem \cite{Au}. 

In this paper, we consider non-integral square functions, which we will denote by $S$, in a broader generality as defined, for example, in \cite{BBR}*{pp. 4--5}. These square functions are not necessarily bounded in the whole range of $L^p$ spaces, with $1 < p < \infty$. Typically such square functions are bounded from $L^p$ into $L^p$ for $p_0 < p < q_0$ and satisfy a weak estimate at the endpoint $p_0$ for some $p_0, q_0$ such that $1 \leq p_0 < 2 < q_0 \leq \infty$. 

We are interested in finding the sharp constant for the weighted estimate
$$
S:\, L^2(w) \to  L^{2,\infty}(w),
$$
which is known to hold for such square functions when $w\in A_{2/p_0}\cap  \operatorname{RH}_{q_0^*}$, where $q_0^*$ is the dual exponent of $q_0/2$. See \cite{AM3} for the qualitative strong estimate and \cite{BBR} for the quantitative one, from which the weak one follows trivially. 

A weight $w$ is a positive locally integrable function; we say that a weight $w$ is in the Muckenhoupt $A_p$ class for $1<p<\infty$, and we denote it by $w\in A_p$, if
\begin{equation}\label{e.ap}
  [w]_{A_p}:=\sup_{Q \,\,\text{cube }} \dashint_{Q} w \D{\mu}
  \left(\dashint_{Q} w^{1 - p'} \D{\mu} \right)^{p-1}<\infty,
\end{equation}
where $p'=p/(p-1)$ is the dual exponent of $p$.
We say that a weight $w$ belongs to the reverse H\"older class $\operatorname{RH}_p$ for $p>1$ if 
\begin{equation*}
  [w]_{\operatorname{RH}_p}:= \sup_{Q \,\,\text{cube }} \left(\dashint_{Q} w^p \D{\mu} \right)^{1/p} \left( \dashint_{Q} w \D{\mu} \right)^{-1}<\infty.
\end{equation*}

Our main result reads as follows:

\begin{theorem} 
\label{thm:main}
Let $p_0<2<q_0$. For any function $f\in L^2(w)$, and weight $w\in A_{\frac{2}{p_0}}\cap \operatorname{RH}_{(\frac{q_0}{2})'}$ we have
$$
\|S(f)\|_{L^{2,\infty}(w)} \le C [w]_{A_{2/p_0}}^{1/2} [w]_{ \operatorname{RH}_{q_0^*}}^{1/2}\eta_{\varepsilon} ([w]_{ \operatorname{RH}_{q_0^*}}, [w]_{A_{\infty}})^{1/2}\|f\|_{L^2(w)},
$$
where  
\begin{equation}\label{e.excess}
\eta_{\varepsilon} ([w]_{ \operatorname{RH}_{q_0^*}}, [w]_{A_{\infty}}) := [w]_{ \operatorname{RH}_{q_0^*}}^{1-\frac{\varepsilon}{q_0^*(q_0^*+\varepsilon-1)}}   \frac{1}{\varepsilon} \Big(   \frac{1}{\varepsilon} +\log[w]_{A_{\infty}}\Big)
\end{equation}
and $0 < \varepsilon \leq \frac{q_0^*}{2^{d+1} [w^{q_0^*}]_{A_{\infty}}-1 }$.
The estimate is uniform in the weight.
\end{theorem}

As an immediate corollary by taking $\varepsilon = \frac{1}{2^{d+1} [w^{q_0^*}]_{A_{\infty}}} $, we have the following.
\begin{corollary}
\label{cor:explicit}
  Let $p_0<2<q_0$. For any function $f\in L^2(w)$, and weight $w\in A_{\frac{2}{p_0}}\cap \operatorname{RH}_{(\frac{q_0}{2})'}$ we have
$$
\|S(f)\|_{L^{2,\infty}(w)} \le C [w]_{A_{2/p_0}}^{1/2} [w]_{ \operatorname{RH}_{q_0^*}}^{1/2}\eta([w]_{ \operatorname{RH}_{q_0^*}}, [w^{q_0^*}]_{A_{\infty}},[w]_{A_{\infty}})^{1/2}\|f\|_{L^2(w)},
$$
where  
\begin{equation}\label{e.excessbis}
\eta ([w]_{ \operatorname{RH}_{q_0^*}},[w^{q_0^*}]_{A_{\infty}}, [w]_{A_{\infty}}) := [w]_{ \operatorname{RH}_{q_0^*}}^{1-\gamma}   [w^{q_0^*}]_{A_{\infty}} \big(  [w^{q_0^*}]_{A_{\infty}} +\log[w]_{A_{\infty}}\big),
\end{equation}
for some small $\gamma<\frac14$. The estimate is uniform in the weight.  
\end{corollary}

Our result is the first quantitative weighted weak $L^p$ estimate for non-integral square functions. Even for non-integral Riesz transforms, the quantitative weighted weak $L^p$ estimates are unknown, see \cite{BFP} for the quantitative strong type estimates. We do not know how to tackle the latter with our current techniques. There is a different type of result for non-integral square functions at the endpoint $p_0$ due to Nieraeth and Stockdale \cite{NS}, where weighted weak $L^{p_0}$ estimates are replaced by mixed type estimates.

In order to understand the estimate obtained in Theorem \ref{thm:main}, where the contribution of $ [w]_{A_{2/p_0}}$ seems to be smaller than the contribution from $[w]_{ \operatorname{RH}_{q_0^*}}$, we need to say a few words about the proof. We first reduce the problem by use of a sparse domination form for the operator developed by one of the authors in \cite{BBR}. Then, the proof strategy contemplates a decoupling between the $A_{2/p_0}$ and the $\operatorname{RH}_{q_0^*}$ roles of the weight at hand. The need for quantitative estimates of Gehring's lemma for weights in the reverse H\"older class (see \cite{G}) appears very naturally. Gehring's celebrated work asserts that a weight in a reverse H\"older class $\operatorname{RH}_q$, for some $1<q<\infty$, must also be in the class $\operatorname{RH}_{q+\varepsilon}$ for some small $\varepsilon$. 

We can also provide weighted weak $L^p$ estimates for the non-integral square functions when $p_0 < p < q_0$, $p \neq 2$, using a quantitative extrapolation result in this setting. It would be of interest to know if there is an extrapolation argument that also decouples the Muckenhoupt and the reverse H\"older roles of the weight. It would most likely lead to a refinement of our estimates.

Section \ref{sec:proof} contains the proof of the main theorem and the comparison with strong type results. In Section \ref{sec:pneq2} we consider the case $p\neq 2$.

\subsection*{Acknowledgments} We thank Kangwei Li for reading the first version of the manuscript, and for providing valuable comments and suggestions.

\section{Proof of Theorem \ref{thm:main}}
\label{sec:proof}

\subsection{Sharp reverse H\"older inequality}

We start by establishing some reverse H\"older estimates that are crucial to the problem at hand. They can be seen as a quantitative version of Gehring's lemma and they are deduced from sharp reverse H\"older estimates considered for $w\in A_{\infty}$.

Below and in the following, we use the notation $w(G):=\int_Gw(x)\,dx$. Moreover, throughout this section, $Q\subset \mathbb{R}^d$ denotes a fixed cube with sides parallel to the coordinate axes, and $M$ will denote the maximal function restricted to the dyadic subcubes of $Q$.  

\begin{proposition}
\label{prop:srH}
    Assume that $w=w 1_Q$. If $w\in \operatorname{RH}_{q_0^*}$, for $0 < \varepsilon \leq \frac{q_0^*}{\tau_{d} [w^{q_0^*}]_{A_{\infty}}-1 }$ we have
    \begin{equation*}        
    \frac{1}{|Q|} \int_Q w^{q_0^* + \varepsilon} \, dx \le 2[w]_{\operatorname{RH}_{q_0^*}}^{q_0^*+\varepsilon}\Big( \frac{1}{|Q|} \int_Q w \, dx \Big)^{q_0^* + \varepsilon}.
    \end{equation*}
\end{proposition}

\begin{proof}
We use the well-known fact that $w\in \operatorname{RH}_{q_0^*}$ if and only if $w^{q_0^*}\in A_{\infty}$ (see \cite{SW}) and the sharp reverse H\"older inequality from \cite{HPR}*{Theorem 2.3} for the weight $w^{q_0^*}$ to obtain
 \begin{equation}\label{e.rhrela}      
    \frac{1}{|Q|} \int_Q w^{q_0^* + \varepsilon} \, dx \le 2\Big( \frac{1}{|Q|} \int_Q w^{q_0^*} \, dx \Big)^{1 + \varepsilon/q_0^*}.
    \end{equation}
for $0<\varepsilon \leq \frac{q_0^*}{2^{d+1}[w^{q_0^*}]_{A_{\infty}}-1} $. On the right hand side of \eqref{e.rhrela} we use the definition of $w\in \operatorname{RH}_{q_0^*}$, which gives us the desired estimate.
\end{proof}

We now apply Proposition \ref{prop:srH} to get the following lemma.

\begin{lemma}
Let $w\in \operatorname{RH}_{q_0^*}$. 
For  $0 < \varepsilon \leq \frac{q_0^*}{2^{d+1} [w^{q_0^*}]_{A_{\infty}}-1 }$ we define
\begin{equation}
\label{eq:alpha}
\theta:=\frac{q_0^*+\varepsilon-1}{q_0^*-1}.
\end{equation}
Then for $E\subset Q$  we have
\begin{equation}\label{e.claimneedsproving}
\frac{w^{q_0^*}(E)}{w^{q_0^*}(Q)}\lesssim [w]_{\operatorname{RH}_{q_0^*}}^{(q_0^*+\varepsilon)/\theta}\Big(\frac{w(E)}{w(Q)}\Big)^{1/\theta'}.
\end{equation}
\end{lemma}

\begin{proof}
First we observe that, by using Proposition \ref{prop:srH}, 
\begin{align}
\label{eq:aux2}
\notag\Big(\frac{1}{w(Q)}\int_Qw^{(q_0^*-1)\theta}\,dw\Big)^{1/\theta}&=\Big(\frac{|Q|}{w(Q)}\Big)^{1/\theta}\Big(\frac{1}{|Q|}\int_Qw^{q_0^*+\varepsilon}\,dx\Big)^{1/\theta}\\
\notag& \leq 2^{1/\theta}[w]_{\operatorname{RH}_{q_0^*}}^{(q_0^*+\varepsilon)/\theta}\Big(\frac{|Q|}{w(Q)}\Big)^{1/\theta}\Big(\frac{1}{|Q|}\int_Qw\,dx\Big)^{\frac{q_0^*+\varepsilon}{\theta}}\\
&=2^{1/\theta} [w]_{\operatorname{RH}_{q_0^*}}^{(q_0^*+\varepsilon)/\theta}\langle w\rangle_Q^{q_0^*-1} \\
& \leq 2^{1/\theta}[w]_{\operatorname{RH}_{q_0^*}}^{(q_0^*+\varepsilon)/\theta} \frac{w^{q_0^*}(Q)}{w(Q)}.
\end{align}
Using this and H\"older's inequality we have
\begin{align*}
\frac{w^{q_0^*}(E)}{w(Q)}=\frac{1}{w(Q)}\int_Ew^{q_0^*-1}\,dw&\le \Big(\frac{1}{w(Q)}\int_Ew^{(q_0^*-1)\theta}\,dw\Big)^{1/\theta}\Big(\frac{1}{w(Q)}\int_Q1_E\,dw\Big)^{1/\theta'}\\
&\leq 2^{1/\theta}[w]_{\operatorname{RH}_{q_0^*}}^{(q_0^*+\varepsilon)/\theta} \frac{w^{q_0^*}(Q)}{w(Q)}\Big(\frac{w(E)}{w(Q)}\Big)^{1/\theta'}.
\end{align*}
Hence $\eqref{e.claimneedsproving}$ holds.

\end{proof}

\begin{remark}
We conjecture the following
\begin{conjecture}
\label{conj:gehring}
    If $w\in \operatorname{RH}_{p}$ and $1<p<\infty$, for $0 < \varepsilon \lesssim_{d,p} \frac{1}{ [w]_{\operatorname{RH}_{p}}^{p} }$ we have
    \begin{equation*}        
    \frac{1}{|Q|} \int_Q w^{p + \varepsilon} \, dx \le 2[w]_{\operatorname{RH}_{p}}^{p+\varepsilon}\Big( \frac{1}{|Q|} \int_Q w \, dx \Big)^{p + \varepsilon}.
    \end{equation*}
\end{conjecture}
The conjecture is based on computations, assisted by the software Mathematica, which deal with the expression in \cite{DW}*{Theorem 3}; such an expression is valid for all $1<p<\infty$ and gives the sharp value of $\varepsilon$ in terms of $[w]_{\operatorname{RH}_{p}}$ and dimension $d=1$.  Unfortunately, the value of $\varepsilon$ is not explicit in the characteristic of the weight (in \cite{DW}*{Theorem 3}, the parameter $\varepsilon$ is defined implicitly in terms of their parameter $t^*$), and therefore not useful for the purpose of this paper. See also \cite{Po}*{Theorem 3.2}.

If Conjecture \ref{conj:gehring} were true, the choice $\varepsilon = c_{d,q_0^*} \frac{1}{ [w]_{\operatorname{RH}_{q_0^*}}^{q_0^*} }$  in Theorem \ref{thm:main} would yield, instead of Corollary \ref{cor:explicit}, the following inequality
\begin{equation}
\label{eq:alternative}
\|S(f)\|_{L^{2,\infty}(w)} \le C [w]_{A_{2/p_0}}^{1/2} [w]_{ \operatorname{RH}_{q_0^*}}^{1/2}[w]_{ \operatorname{RH}_{q_0^*}}^{1-\gamma}   [w]_{\operatorname{RH}_{q_0^*}}^{q_0^*/2} \big(  [w]_{\operatorname{RH}_{q_0^*}}^{q_0^*} +\log[w]_{A_{\infty}}\big)^{1/2}\|f\|_{L^2(w)}.
\end{equation}
Let us recall the relation in \cite{KS}: if  $w\in \operatorname{RH}_{q_0^*}$, then 
\begin{equation}
\label{e.cuantrhnorm}
\frac{[w^{q_0^*}]_{A_{\infty}}^{1/q_0^*}}{[w]_{A_{\infty}}}\leq [w]_{RH_{q_0^*}}\leq [w]_{A_{\infty}}^{1/q_0^*}.
\end{equation}
Applying \eqref{e.cuantrhnorm} in \eqref{eq:alternative}, we get
$$
\|S(f)\|_{L^{2,\infty}(w)} \le C [w]_{A_{2/p_0}}^{1/2} [w]_{ \operatorname{RH}_{q_0^*}}^{1/2}[w]_{ \operatorname{RH}_{q_0^*}}^{1-\gamma}   [w]_{A_{\infty}}^{1/2} \big(  [w]_{A_{\infty}} +\log[w]_{A_{\infty}}\big)^{1/2}\|f\|_{L^2(w)},
$$
which is a better bound than the one in Corollary \ref{cor:explicit}. The conclusion is that a sharp Gehring lemma with explicit information of $\varepsilon$ in terms of $[w]_{ \operatorname{RH}_{q_0^*}}$ would produce a better quantitative estimate for the non-integral square function.
\end{remark}

\subsection{Some known facts and a sparse domination}

We start by writing some elementary facts as lemmas; their proofs are left to the reader. Let $M_{p_0}$ be the dyadic maximal function  defined using $p_0$ averages over dyadic cubes.

\begin{lemma}
Let $1<p_0<2$. Then
\begin{equation}\label{e.maxpo2}
\lVert M_{p_0}f \rVert_{L^{2,\infty}(w)}\lesssim [w]_{A_{2/p_0}}^{1/2} \lVert f \rVert_{L^{2}(w)},
\end{equation}
where $w\in A_{2/p_0}$ and $[w]_{A_{2/p_0}}$ is its associated constant.
\end{lemma}

The equivalence in items (1)--(3) below follows, for instance, from \cite{GrafakosC}*{p. 80, 1.4.14 (c)}. 

\begin{lemma}
The following are equivalent: for $1<p<\infty$,
\begin{enumerate}
\item $S(\cdot)\,:\, L^p(w) \mapsto  L^{p,\infty}(w),$
\item $S(\sigma\cdot)\,:\, L^p(\sigma) \mapsto  L^{p,\infty}(w),$ where $\sigma=w^{1-p'}$.
\item For every $G\subset \mathbb R^n$, there exists $G'\subset G$, such that $w(G')>w(G)/4$ and 
$$
\langle S(f \sigma)^2, 1_{G'}w\rangle \lesssim \lVert f\rVert _{L^p(\sigma)}^2w(G)^{1-2/p}.
$$
\end{enumerate}
\end{lemma}

We will study the case $p=2$, which corresponds to the critical case for the weak estimate for the non-integral square function. Therefore, given $G \subset \mathbb R^d$, we need to prove that there exists $G'\subset G$ with $w(G')>w(G)/4$ such that 
$$
\langle S(f \sigma)^2, 1_{G'}w\rangle \lesssim \lVert f \rVert_{L^{2}(\sigma)}^{2},
$$
and we are interested in the explicit constant in terms of the weight. 

\begin{definition}
    Consider a system of dyadic cubes $\mathcal{D}$. A collection of dyadic cubes $\mathcal{A}\subseteq \mathcal{D}$ is $1/2$-sparse if there exists a disjoint collection of sets $\{E_Q:Q\in \mathcal{A}\}$ such that for every $Q\subseteq \mathcal{A}$ we have $E_Q\subset Q$ and $|E_Q|>\frac12|Q|$.
\end{definition}

Our starting point is the following sparse domination result contained in \cite{BBR}.
\begin{theorem}{\cite{BBR}*{Theorem 1.7}} 
\label{thm:BBR}
Let $p_0<2<q_0$ and consider the non-integral square function $S$. For any $f,g\in C_c^{\infty}(\mathbb{R}^d)$ there exists a sparse family $\mathcal{A}\subseteq \mathcal{D}$ such that
$$
\langle (Sf)^2,g\rangle \le c\sum_{Q\in \mathcal{A}}\langle |f|\rangle_{p_0, 5Q}^2\langle |g|\rangle_{q_0^*, 5Q}|Q|,
$$
uniformly in $f,g$, where $q_0^*:=\big(\frac{q_0}{2}\big)'$ is the dual exponent of $\frac{q_0}{2}$.
\end{theorem}

 Using the sparse domination result in Theorem \ref{thm:BBR} as a black box to prove Theorem \ref{thm:main}, it is enough to prove that given a collection of sparse cubes $\mathcal A$ and a weight $w\in A_{2/p_0}\cap \operatorname{RH}_{q_0^*}$, where $q_0^*:=\big(\frac{q_0}{2}\big)'$, we have 
\begin{equation}\label{e.mainestimateni}
\sum_{Q\in\mathcal A} \langle f \sigma \rangle^2_{p_0,Q} \langle 1_{G}w  \rangle_{q_0^*, Q}|Q| \lesssim C_0 [w]_{A_{2/p_0}} [w]_{ \operatorname{RH}_{q_0^*}}\eta_{\varepsilon} ([w]_{ \operatorname{RH}_{q_0^*}}, [w]_{A_{\infty}}) \|f\|_{L^2(\sigma)}^2,
\end{equation}
where $\eta_{\varepsilon}$ is given in \eqref{e.excess}.

\subsection{Proof of Theorem \ref{thm:main}}
\label{sub:estimate}

We start with the following consideration: we denote by $M_{{p_0}, \mathcal A}$ the maximal function, where the corresponding supremum is taken over the cubes in the collection $\mathcal A$.  Let  $G$ be a subset of $\mathbb R^d$,  we define $G'$ as
\begin{equation}\label{e.choiceofGpo2}
G':=G \setminus \bigg\{M_{{p_0}, \mathcal A}(f \sigma)>K[w]_{A_{2/p_0}}^{1/2}\frac{\|f\|_{L^{2}(\sigma)}}{\sqrt{w(G)}}  \bigg\},
\end{equation}
where $K$ is a constant to be chosen. In fact, due to \eqref{e.maxpo2}, we get that $w(G')\geq \frac{3}{4} w(G)$ as long as $K$ is sufficiently large.

Due to the choice of $G'$ in \eqref{e.choiceofGpo2}, for all $Q\in \mathcal A$ such that $G'\cap Q \neq \emptyset$, we have
$$
\langle f \sigma  \rangle_{p_0, Q}  \leq K[w]_{A_{2/p_0}}^{1/2}\frac{\|f\|_{L^{2}(\sigma)}}{\sqrt{w(G)}}, \qquad
\langle  1_{G'} \rangle_{Q}^w  \leq 1.
$$
We are now in a position to proceed with the following pigeonholing. For $r,s\geq 0$, and sets $F$, $G'$, we let
\begin{equation}\label{e.pg}
\mathcal A_{r,s}:= \bigg \{Q\in \mathcal A\,: \,\, \langle f\sigma  \rangle_{p_0, Q} \sim K 2^{-r}[w]_{A_{2/p_0}}^{1/2}\frac{\|f\|_{L^{2}(\sigma)}}{\sqrt{w(G)}}\, ,\,\,  \langle  1_{G'}\rangle_{Q}^w \sim 2^{-s} \bigg \},
\end{equation}
where $G'$ is the set described in \eqref{e.choiceofGpo2} and $\langle f\rangle_Q^w:=\frac{1}{w(Q)}\int_Qf(x)w(x)\,dx$. We write $\mathcal A_{r,s}= \cup_{k=1}^{\infty}  \mathcal A_{r,s,k}$, where
$ \mathcal A_{r,s,1}$ is the collection of maximal cubes in  $\mathcal A_{r,s}$ and for $k\geq 2$,  $ \mathcal A_{r,s,k}$  is the collection of maximal cubes in $\mathcal A_{r,s}\setminus \cup_{j=1}^{k-1} \mathcal A_{r,s,j}$.

Let $Q\in \mathcal{A}_{r,s}$. Choosing $E$ to be equal to $G'\cap Q$ in \eqref{e.claimneedsproving} and recalling \eqref{e.pg} we conclude

\begin{align}\label{e.compaverages}
\notag\langle 1_{G'} w \rangle_{q_0^*, Q}= \Big(\frac{1}{|Q|}\int_Q1_{G'}w^{q_0^*}\,dx\Big)^{1/q_0^*}&= \Big(\frac{w^{q_0^*}(G' \cap Q)}{w^{q_0^*}(Q)}\Big)^{1/q_0^*}\Big(\frac{w^{q_0^*}(Q)}{|Q|}\Big)^{1/q_0^*} \\
\notag&\lesssim  \langle w \rangle_{q_0^*, Q}[w]_{\operatorname{RH}_{q_0^*}}^{(q_0^*+\varepsilon)/(\theta q_0^*)} \left (\frac{w(G' \cap Q)}{w(Q)}\right )^{\frac{1}{\theta'q_0^*}}\\
\notag&\lesssim  \langle w\rangle_Q [w]_{\operatorname{RH}_{q_0^*}}^{1+(q_0^*+\varepsilon )/(\theta q_0^*)}(\langle 1_{G'} \rangle_Q^w )^{\frac{1}{\theta'q_0^*}}\\
&\simeq   [w]_{\operatorname{RH}_{q_0^*}}^{2-\frac{\varepsilon}{q_0^*(q_0^*+\varepsilon-1)}}2^{-\frac{s}{\theta'q_0^*}}\langle w\rangle_Q.
\end{align}

\color{black}

We are now in a position to prove \eqref{e.mainestimateni}. We will need two estimates. For the first one, notice that 
\begin{align*}
\sum_{Q\in\mathcal A_{r,s,1}}w(Q) &\leq \sum_{Q\in\mathcal A_{r,s,1}}w(\{x\in Q:M_w(1_{G'})>2^{-s}\}) \\ &\leq w(\{x:M_w(1_{G'})>2^{-s}\}) \leq 2^sw(G'),
\end{align*}
because $M_w$, which is the maximal function defined with respect to the measure $w(x)dx$, maps $L^1(w)$ to $ L^{1,\infty}(w)$ with constant independent of $w$. On the other hand, if $Q_0\in \mathcal A_{r,s,1}$,
\begin{align*}
\sum_{\substack{Q\in \mathcal{A}_{r,s}\\ Q\subset Q_0} } w(Q) & = \sum_{\substack{Q\in \mathcal{A}_{r,s}\\ Q\subset Q_0 }} \frac{w(Q)}{|Q|}|Q|  \leq \sum_{\substack{Q\in \mathcal{A}_{r,s}\\ Q\subset Q_0} } 2\frac{w(Q)}{|Q|}|E(Q)|\\
&\leq 2  \sum_{\substack{Q\in \mathcal{A}_{r,s}\\ Q\subset Q_0 }} \int_{E(Q)} M(w1_{Q_0}) \leq 2 \int_{Q_0} M(w1_{Q_0})\leq 2[w]_{A_{\infty}}w(Q_0).
\end{align*}
Hence,
$$
\sum_{Q\in\mathcal A_{r,s}}w(Q)\leq [w]_{A_{\infty}}\sum_{Q\in\mathcal A_{r,s,1}}w(Q)\leq  [w]_{A_{\infty}}2^sw(G').
$$
Then we have 
\begin{align}\label{estimate1ni}
\notag&\sum_{Q\in\mathcal A_{r,s}} \langle f \sigma \rangle_{p_0, Q}^2 \langle 1_{G'} w \rangle_{q_0^*, Q}|Q|  \simeq 2^{-2r}[w]_{A_{2/p_0}} \frac{\| f\| _{L^{2}(\sigma)}^{2}}{w(G')}\sum_{Q\in\mathcal A_{r,s}} \langle 1_{G'} w \rangle_{q_0^*, Q}|Q|  \\
\notag&\quad\lesssim 2^{-2r}[w]_{A_{2/p_0}}[w]_{\operatorname{RH}_{q_0^*}}^{2-\frac{\varepsilon}{q_0^*(q_0^*+\varepsilon-1)}}2^{-\frac{s}{\theta'q_0^*}}\frac{\| f\| _{L^{2}(\sigma)}^{2}}{w(G')}\sum_{Q\in\mathcal A_{r,s}}  \langle w\rangle_Q|Q| \\
&\quad \leq  2^{-2r}[w]_{A_{2/p_0}}[w]_{\operatorname{RH}_{q_0^*}}^{2-\frac{\varepsilon}{q_0^*(q_0^*+\varepsilon-1)}}2^{s\big(1-\frac{1}{\theta' q_0^*}\big)}[w]_{A_{\infty}}\| f\| _{L^{2}(\sigma)}^{2}
\end{align}
\color{black}
where we have used \eqref{e.compaverages}.

We now produce the second estimate.  For $Q \in \mathcal{A}_{r,s}$, define a exceptional set $E_Q \subseteq Q$ relative to $\mathcal{A}_{r,s}$ by
$$
E_Q : = Q\  \Big\backslash \bigcup_{\substack{ Q' \subsetneq Q \\ Q' \in \mathcal{A}_{r,s}}} Q'.
$$
The sparsity condition and the definition of $ \mathcal{A}_{r,s}$ implies that (see \cite{DLR}*{(4.1)}) 

\begin{equation}
\label{eq:sparsity}
\langle f \sigma  \rangle_{p_0, Q}\sim \langle1_{E_Q} f \sigma  \rangle_{p_0,Q}.
\end{equation}
Let $\varphi(p_0):= \frac{p_0}{2-p_0}$. Thus, 
\begin{align} \label{estimate2ni}
\notag&\sum_{Q\in\mathcal A_{r,s}} \langle f \sigma \rangle_{p_0, Q}^2  \langle 1_{G'} w \rangle_{q_0^*, Q}|Q|\lesssim 2^{\frac{1}{\theta q_0^*}} [w]_{\operatorname{RH}_{q_0^*}}^{2-\frac{\varepsilon}{q_0^*(q_0^*+\varepsilon-1)}}2^{-\frac{s}{\theta'q_0^*}}\sum_{Q\in\mathcal A_{r,s}} \langle f \sigma \rangle_{p_0, Q}^2 \langle w \rangle_Q |Q|\\
 \notag&\quad\simeq 2^{\frac{1}{\theta q_0^*}} [w]_{\operatorname{RH}_{q_0^*}}^{2-\frac{\varepsilon}{q_0^*(q_0^*+\varepsilon-1)}}2^{-\frac{s}{\theta'q_0^*}}[w]_{A_{2/p_0}}\sum_{Q\in\mathcal A_{r,s}} \langle f\sigma  \rangle_{p_0, Q}^2  \langle \sigma \rangle_{\varphi(p_0), Q}^{-1}|Q|\\
\notag&\quad\leq 2^{\frac{1}{\theta q_0^*}} [w]_{\operatorname{RH}_{q_0^*}}^{2-\frac{\varepsilon}{q_0^*(q_0^*+\varepsilon-1)}}2^{-\frac{s}{\theta'q_0^*}} [w]_{A_{2/p_0}} \sum_{Q\in\mathcal A_{r,s}} \langle1_{E_Q} f \sigma  \rangle_{p_0,Q}^2 \langle \sigma \rangle_{\varphi(p_0), Q}^{-1}|Q|\\
\notag&\quad =  2^{\frac{1}{\theta q_0^*}} [w]_{\operatorname{RH}_{q_0^*}}^{2-\frac{\varepsilon}{q_0^*(q_0^*+\varepsilon-1)}}2^{-\frac{s}{\theta'q_0^*}}[w]_{A_{2/p_0}} \sum_{Q\in\mathcal A_{r,s}} \Big(\frac{1}{|Q|}\int_Q1_{E_Q}f^{p_0} \sigma^{p_0/2} \sigma^{p_0/2}\,dx\Big)^{2/p_0} \langle \sigma \rangle_{\varphi(p_0), Q}^{-1}|Q|\\
 \notag&\quad\leq 2^{\frac{1}{\theta q_0^*}} [w]_{\operatorname{RH}_{q_0^*}}^{2-\frac{\varepsilon}{q_0^*(q_0^*+\varepsilon-1)}}2^{-\frac{s}{\theta'q_0^*}}[w]_{A_{2/p_0}}  \sum_{Q\in\mathcal A_{r,s}}\int_Q 1_{E_Q}f^{2} \sigma\,dx\\
&\quad \leq  2^{\frac{1}{\theta q_0^*}} [w]_{\operatorname{RH}_{q_0^*}}^{2-\frac{\varepsilon}{q_0^*(q_0^*+\varepsilon-1)}}2^{-\frac{s}{\theta'q_0^*}}[w]_{A_{2/p_0}} \|f\|_{L^{2}(\sigma)}^{2},
\end{align}
where we used \eqref{e.compaverages} in the first inequality, sparsity \eqref{eq:sparsity} in the second inequality, and H\"older inequality in the third one.

Therefore, in view of \eqref{estimate1ni} and \eqref{estimate2ni} we are led to estimate the sum (recall that $\theta$ is defined in \eqref{eq:alpha}, and there is an implicit $\varepsilon$ depending on $[w^{q_0^*}]_{A_{\infty}} $ in Proposition~\ref{prop:srH}) 
\begin{align}
\label{eq:todojunto}
\notag&  \sum_{r,s\ge 0} \sum_{Q\in\mathcal A_{r,s}} \langle f \sigma \rangle_{p_0, Q}^2  \langle 1_{G'} w \rangle_{q_0^*, Q}|Q| \\
\notag& \lesssim   \sum_{r,s\ge 0}2^{\frac{1}{\theta q_0^*}} [w]_{\operatorname{RH}_{q_0^*}}^{2-\frac{\varepsilon}{q_0^*(q_0^*+\varepsilon-1)}}[w]_{A_{2/p_0}} 2^{-\frac{s}{\theta'q_0^*}}\min\{2^{-2r}2^s[w]_{A_{\infty}},1\}\|f\|_{L^{2}(\sigma)}^{2} \\
\notag& = 2^{\frac{1}{\theta q_0^*}} \|f\|_{L^{2}(\sigma)}^{2} [w]_{\operatorname{RH}_{q_0^*}}^{2-\frac{\varepsilon}{q_0^*(q_0^*+\varepsilon-1)}}[w]_{A_{2/p_0}}\sum_{s\ge 0} 2^{-\frac{s}{\theta'q_0^*}}\\
\notag&\qquad \times\Big(\sum_{r\ge\frac12(s+\log[w]_{A_{\infty}})}2^{-2r}2^s[w]_{A_{\infty}}
+\sum_{r\le\frac12(s+\log[w]_{A_{\infty}})}1\Big)\\
\notag& = 2^{\frac{1}{\theta q_0^*}} \|f\|_{L^{2}(\sigma)}^{2} [w]_{A_{2/p_0}} [w]_{\operatorname{RH}_{q_0^*}}^{2-\frac{\varepsilon}{q_0^*(q_0^*+\varepsilon-1)}} \\
&\qquad\times\Big(\sum_{s\ge 0} 2^{s-\frac{s}{\theta'q_0^*}}\sum_{r\ge\frac12(s+\log[w]_{A_{\infty}})}2^{-2r}[w]_{A_{\infty}}
+\sum_{s\ge0}2^{-\frac{s}{\theta'q_0^*}}(s + \log[w]_{A_{\infty}})\Big).
\end{align}
\color{black}

On the other hand, 
$$
\sum_{s\ge 0} 2^{s-\frac{s}{\theta'q_0^*}}\sum_{r\ge\frac12(s+\log[w]_{A_{\infty}})}2^{-2r}[w]_{A_{\infty}}\simeq \sum_{s\ge 0} 2^{-\frac{s}{\theta'q_0^*}}= \frac{2^{\frac{1}{\theta' q_0^*}}}{2^{\frac{1}{\theta' q_0^*}}-1}, 
$$
and 
$$
\sum_{s\ge0}2^{-\frac{s}{\theta'q_0^*}}s \simeq \frac{2^{\frac{1}{\theta' q_0^*}}}{(2^{\frac{1}{\theta' q_0^*}}-1)^2},\qquad \sum_{s\ge0}2^{-\frac{s}{\theta'q_0^*}}\log[w]_{A_{\infty}}\simeq\frac{2^{\frac{1}{\theta' q_0^*}}}{2^{\frac{1}{\theta' q_0^*}}-1}\log[w]_{A_{\infty}}.
$$
Hence, plugging the above into \eqref{eq:todojunto}, we obtain the following estimate 
\begin{align*}
&\sum_{Q\in\mathcal A} \langle f \sigma \rangle^2_{p_0,Q} \langle 1_{G}w  \rangle_{q_0^*, Q}|Q|\\ 
& \qquad \lesssim  2^{\frac{1}{\theta q_0^*}} [w]_{A_{2/p_0}}  [w]_{\operatorname{RH}_{q_0^*}}^{2-\frac{\varepsilon}{q_0^*(q_0^*+\varepsilon-1)}}\frac{2^{\frac{1}{\theta' q_0^*}}}{2^{\frac{1}{\theta' q_0^*}}-1}\bigg(\frac{1}{2^{\frac{1}{\theta' q_0^*}}-1}+1+\log[w]_{A_{\infty}}\bigg) \|f\|_{L^{2}(\sigma)}^{2}\\
& \qquad \lesssim 2^{\frac{1}{\theta q_0^*}} [w]_{A_{2/p_0}} [w]_{\operatorname{RH}_{q_0^*}}^{2-\frac{\varepsilon}{q_0^*(q_0^*+\varepsilon-1)}}\frac{2^{\frac{1}{\theta' q_0^*}}}{2^{\frac{1}{\theta' q_0^*}}-1}\bigg(\frac{2^{\frac{1}{\theta' q_0^*}}}{2^{\frac{1}{\theta' q_0^*}}-1}+\log[w]_{A_{\infty}}\bigg) \|f\|_{L^{2}(\sigma)}^{2}
\end{align*}
Recalling our choice of $\theta$, we have
$$
\frac{1}{\theta' q_0^*}=\frac{\varepsilon}{q_0^*(q_0^*+\varepsilon-1)},
$$
which is less than one, so we have
$$
\frac{2^{\frac{1}{\theta' q_0^*}}}{2^{\frac{1}{\theta' q_0^*}}-1}\lesssim \theta' q_0^* = \frac{q_0^*(q_0^*+\varepsilon-1)}{\varepsilon}\sim \frac{1}{\varepsilon}. 
$$

Altogether, we have obtained 
\begin{equation*}
\sum_{Q\in\mathcal A} \langle f \sigma \rangle^2_{p_0,Q} \langle 1_{G'}w  \rangle_{q_0^*, Q}|Q|  \lesssim [w]_{A_{2/p_0}} [w]_{ \operatorname{RH}_{q_0^*}}^{2-\frac{\varepsilon}{q_0^*(q_0^*+\varepsilon-1)}}  \frac{1}{\varepsilon}\bigg(\frac{1}{\varepsilon}+\log[w]_{A_{\infty}}\bigg) \|f\|_{L^2(\sigma)}^2,
\end{equation*}
where $0<\varepsilon \leq \frac{q_0^*}{2^{d+1} [w^{q_0^*}]_{A_{\infty}}-1}$. The proof of Theorem \ref{thm:main} is complete.

\subsection{Comparing with strong-type results}

In \cite{BBR}*{Theorem 1.8} the following strong-type quantitative weighted inequality was deduced from the sparse domination in Theorem~\ref{thm:BBR}.

\begin{theorem}{\cite{BBR}*{Theorem 1.8}} 
\label{thm:BBR2}
Let $p_0<2<q_0$. For any sparse family $\mathcal{A}\subset \mathcal{D}$, functions $f,g\in L_{\operatorname{loc}}^1(\mathbb{R}^d)$, $q\in (2,q_0)$ and weight $w\in A_{\frac{q}{p_0}}\cap \operatorname{RH}_{(\frac{q_0}{q})'}$ we have
$$
\sum_{Q\in \mathcal{A}}\langle |f|\rangle_{p_0, 5Q}^2\langle |g|\rangle_{q_0^*, 5Q}|Q|\le C_0\big([w]_{A_{\frac{q}{p_0}}}\cdot [w]_{\operatorname{RH}_{(\frac{q_0}{q})'}}\big)^{2\gamma(q)}\|f\|_{L^q(w)}^2\|g\|_{L^{q^*}(\sigma)},
$$
uniformly in the weight and the sparse collection, where 
$$
\gamma(q):=\max\Big\{\frac{1}{q-p_0},\big(\frac{q_0}{q}\big)'\frac{1}{2q_0^*}\Big\}\quad \text{ and } \sigma:=w^{1-q^*}.
$$
The dependence of the above estimate on the weight characteristic is sharp.
\end{theorem}

From Theorem \ref{thm:BBR2}, it can be deduced (see \cite{BBR}*{Corollary 1.9}) that 
$$
\|S\|_{L^2(w)\to L^2(w)}\lesssim \big([w]_{A_{\frac{2}{p_0}}}\cdot [w]_{\operatorname{RH}_{(\frac{q_0}{2})'}}\big)^{\gamma(2)}
$$
and, on the other hand, we have obtained
$$
\|S(f)\|_{L^{2,\infty}(w)} \le C [w]_{A_{2/p_0}}^{1/2} [w]_{ \operatorname{RH}_{q_0^*}}^{1/2}\eta([w]_{ \operatorname{RH}_{q_0^*}}, [w^{q_0^*}]_{A_{\infty}},[w]_{A_{\infty}})^{1/2}\|f\|_{L^2(w)},
$$
where  
\begin{equation*}
\eta ([w]_{ \operatorname{RH}_{q_0^*}},[w^{q_0^*}]_{A_{\infty}}, [w]_{A_{\infty}}) := [w]_{ \operatorname{RH}_{q_0^*}}^{1-\gamma}   [w^{q_0^*}]_{A_{\infty}} \big(  [w^{q_0^*}]_{A_{\infty}} +\log[w]_{A_{\infty}}\big)
\end{equation*}
for some $\gamma<1/4$.

One may wonder whether the exponent in the weak-type estimate we obtain is smaller than the exponent in the existing strong-type estimate. If we focus on the $[w]_{A_{2/p_0}}$ constant and ignore the $\phi$-term in the weak estimate, the power for the strong-type estimate is $\gamma(2)=\frac{1}{2-p_0}$ and the power for the weak-type estimate is $\frac{1}{2}$. Clearly, the weak estimate is smaller than the strong one in terms of the $[w]_{A_{2/p_0}}$ constant. On the other hand, concerning the reverse H\"older constant $[w]_{ \operatorname{RH}_{q_0^*}}$, once again we have power $\frac{1}{2-p_0}$ for the strong estimate but now the power for the weak-type estimate is $1-\gamma/2$. In this case, it is not possible to conclude which bound is smaller.

\section{Results for \texorpdfstring{$p\neq 2$}{p≠2}}
\label{sec:pneq2}

The method of proof we have used for the main theorem seems exclusive to the case $p=2$. One could obtain bounds for the other cases using a quantitative version of a qualitative extrapolation theorem which involves the classes of weights $A_{\frac{q}{p_0}}\cap \operatorname{RH}_{(\frac{q_0}{q})'}$, see \cite{AuscherMartellIGeneral}*{Theorem 4.9}, and which we present in Theorem \ref{thm:Extrapolation} below. Although these bounds, for $p\neq 2$, may be smaller than the strong bounds in certain cases, we cannot assure this in general. We can give the following estimates for the case $p\neq 2$.

\begin{corollary} \label{c.allp}
Let $p_0<q<q_0$. For any functions $f\in L^q(w)$, and weights $w\in A_{\frac{q}{p_0}}\cap \operatorname{RH}_{(\frac{q_0}{q})'}$, we have
\begin{multline*}
\|S(f)\|_{L^{q,\infty}(w)} \lesssim ([w]^{(q_0/q)'}_{A_{q/p_0}}[w]^{(q_0/q)'}_{\operatorname{RH}_{(q_0/q)'}})^{\beta(2,q){\frac{3-\gamma+q_0^*}{ 2q^*_0}}}\\
\times\Big([w]^{(q_0/q)'}_{A_{q/p_0}}[w]^{(q_0/q)'}_{\operatorname{RH}_{(q_0/q)'}}+\frac{1}{q_0^*}\log([w]^{(q_0/q)'}_{A_{q/p_0}}[w]^{(q_0/q)'}_{\operatorname{RH}_{(q_0/q)'}})\Big)^{\frac{\beta(2,q)}{2}}\|f\|_{L^q(w)},
\end{multline*}
for some small $\gamma<1/4$ and $\beta(p,q)\coloneqq \max (1, \frac{(q_{0} - q)(p -
      p_{0})}{(q_{0} - p)(q - p_{0})})$.
The estimate is uniformly in the weight.
\end{corollary}

\subsection{Proof of Corollary \ref{c.allp}}

For the proof, we use extrapolation. First, let us recall some properties of the classes of weights $A_p$ and $\operatorname{RH}_s$.

\begin{lemma}
  \label{lem:WeightProperties}
  The following properties of the weight classes $A_{p}$ and $\operatorname{RH}_{q}$ are true.
  \begin{enumerate}
  \item[(i)] For $p \in (1,\infty)$, a weight $w$ will be contained in
    the class $A_{p}$ if and only if $w^{1 - p'} \in A_{p'}$. Moreover,
    \begin{equation*}
      \big[w^{1-p'}\big]_{A_{p'}} = \big[w\big]^{p' - 1}_{A_{p}}.
    \end{equation*}
  \item[(ii)] For $q \in [1,\infty]$ and $s \in [1,\infty)$, a weight
    $w$ will be contained in $A_{q} \cap \operatorname{RH}_{s}$ if and only if $w^{s}
    \in A_{s(q - 1) + 1}$. Moreover,
    \begin{equation}
    \label{eq:aqrh}
      \max\{[w]^{s}_{A_{q}} , [w]^{s}_{\operatorname{RH}_{s}}\} \le [w^{s}]_{A_{s(q - 1) + 1}}    
      \leq [w]^{s}_{A_{q}} [w]^{s}_{\operatorname{RH}_{s}}.
    \end{equation} 

  \end{enumerate}
\end{lemma}

For $1\leq p_{0} < 2 < q_{0} \leq \infty$ and  $p \in (p_{0},q_{0})$ define
$$
\phi(p) := \Big(\frac{q_{0}}{p}\Big)' \Big(\frac{p}{p_{0}} - 1\Big) + 1.
$$
The dependence of $\phi$ on $p_{0}$ and $q_{0}$ will be kept
implicit. From the previous lemma, we get that a weight $w$ will be contained in the class
$A_{\frac{p}{p_{0}}} \cap \operatorname{RH}_{(\frac{q_{0}}{p})'}$ if and only if
$w^{(\frac{q_{0}}{p})'}$ is contained in $A_{\phi(p)}$ and it will be
true that
\begin{equation}
  \label{eqtn:WeightProperty}
[w^{(\frac{q_{0}}{p})'}]_{A_{\phi(p)}} \leq
\Big([w]_{A_{\frac{p}{p_{0}}}}  [w]_{\operatorname{RH}_{(\frac{q_{0}}{p})'}}\Big)^{(\frac{q_{0}}{p})'}.
\end{equation}

A restricted range extrapolation result presented in \cite{AuscherMartellIGeneral} can be used to obtain
$L^{p}(w)$-boundedness for the full range of $p \in (p_{0},q_{0})$ and
$w \in A_{\frac{p}{p_{0}}} \cap \operatorname{RH}_{(\frac{q_{0}}{p})'}$ directly
from the $L^{q}(w)$-boundedness for all $w \in A_{\frac{q}{p_{0}}} \cap \operatorname{RH}_{(\frac{q_{0}}{q})'}$ of a single index
$q \in (p_{0},q_{0})$. In the result in \cite{AuscherMartellIGeneral}
the dependence of the bound on the weight characteristic
$[w^{(\frac{q_{0}}{p})'}]_{\phi(p)}$ is not stated. Through 
inspection of the proof, by tracing the relevant constants, it is not difficult to see that the
extrapolation result has the following sharp dependence on the
weight characteristic. 

\begin{theorem}[Sharp Restricted Range Extrapolation {\cite{AuscherMartellIGeneral}*{Theorem 4.9}}] 
  \label{thm:Extrapolation} 
  Let $0 < p_{0} < q_{0} \leq \infty$. Let $\mathcal{F} $ denote a family of ordered pairs of non-negative,
measurable functions $(f, g)$. Suppose that there exists an increasing function $\varphi$ and 
  $p$ with $p_{0} \leq p < q_{0}$
  such that for $(f,g) \in \mathcal{F}$,  
  \begin{equation}
 \label{extrapol-q}
    \|f\|_{L^{p}(w)} \leq C \varphi([w^{(\frac{q_{0}}{p})'}]_{A_{\phi(p)}}) \|g\|_{L^{p}(w)}
    \quad \text{ for all } w\in A_{\frac{p}{p_{0}}} \cap \operatorname{RH}_{(\frac{q_{0}}{p})'},
  \end{equation}
  for some $\alpha > 0$ and $C > 0$ independent of the weight.
  Then, for all $p_{0} < q < q_{0}$ and $(f,g) \in \mathcal{F}$ we have
  
  \begin{equation*}
    \|f\|_{L^{q}(w)}
    \leq C' \varphi([w^{(\frac{q_{0}}{q})'}]_{A_{\phi(q)}})^{\beta(p,q)} \|g\|_{L^{q}(w)}
    \quad \text{ for all } w \in A_{\frac{q}{p_{0}}} \cap \operatorname{RH}_{(\frac{q_{0}}{q})'},
  \end{equation*} 
  where $\beta(p,q) \coloneqq \max (1, \frac{(q_{0} - q)(p -
      p_{0})}{(q_{0} - p)(q - p_{0})})$ and $C' > 0$ is independent
  of the weight.
\end{theorem}

 We remark that the quantitative version of the extrapolation theorem in \cite{BBR}*{Theorem 2.15} is slightly less general than Theorem \ref{thm:Extrapolation}.
\begin{corollary}\label{corol:extrapol-weak}
Let $0<p_0<q_0\le \infty$. Suppose that there exists an increasing function $\varphi$ and $p$ with $p_0\le
p\le q_0$, and $p<\infty$ if $q_0=\infty$, such that for $(f,g)\in
\mathcal{F}$,
\begin{equation}\label{extrapol-p:weak}
\|f\|_{L^{p,\infty}(w)}
\le
 C \varphi([w^{(\frac{q_{0}}{p})'}]_{A_{\phi(p)}})\|g\|_{L^{p}(w)}
\quad
\mbox{for all }w\in A_{\frac{p}{p_0}}\cap
\operatorname{RH}_{(\frac{q_0}{p})'}.
\end{equation}
Then, for all $p_0<q<q_0$ and $(f,g)\in\mathcal{F}$ we have
\begin{equation}\label{extrapol-q:weak}
\|f\|_{L^{q,\infty}(w)}
\le
 C' \varphi([w^{(\frac{q_{0}}{q})'}]_{A_{\phi(q)}})^{\beta(p,q)} \|g\|_{L^{q}(w)}
\quad
\mbox{for all }w\in A_{\frac{q}{p_0}}\cap
\operatorname{RH}_{(\frac{q_0}{q})'}.
\end{equation}
\end{corollary}

\begin{proof}
The proof is by now well-known, see e.g. \cites{AuscherMartellIGeneral,GM}. Given $(f,g)\in \mathcal{F}$
and  any $\lambda>0$, define a pair of functions $(f_\lambda,g)$
where $f_\lambda=\lambda\,\chi_{E_\lambda(f)}$   and
$E_\lambda(f)=\{f>\lambda\}$. By using \eqref{extrapol-p:weak} we have
$$
\|f_\lambda\|_{L^p(w)}
=
\lambda\,w(E_\lambda(f))^\frac1p
\le
\sup_\lambda \lambda\,w(E_\lambda(f))^\frac1p
=
\|f\|_{L^{p,\infty}(w)}
\le
 C \varphi([w^{(\frac{q_{0}}{q})'}]_{A_{\varphi(q)}})\|g\|_{L^{p}(w)}
$$
for all $w\in A_{\frac{p}{p_0}}\cap \operatorname{RH}_{(\frac{q_0}{p})'}
$.  Now we apply Theorem \ref{thm:Extrapolation} to the family
$\widetilde{\mathcal{F}}$ of  pairs $(f_\lambda,g)$, which satisfy \eqref{extrapol-q} with
$C$ independent of $\lambda$. After taking the supremum on
$\lambda>0$, we obtain \eqref{extrapol-q:weak}, as desired.
\end{proof}

\begin{proof}[Proof of Corollary \ref{c.allp}]
We have, by Corollary \ref{cor:explicit}, \eqref{eq:aqrh}, and the fact that $[w]_{A_{\infty}}\le [w]_{A_p}$ for all $p\in [1,\infty)$ (see, e.g. \cite{HP}),
\begin{align}
\label{eq:loss}
\notag &\|S(f)\|_{L^{2,\infty}(w)} \lesssim[w]_{A_{2/p_0}}^{1/2} [w]_{ \operatorname{RH}_{q_0^*}}^{1/2}[w]_{ \operatorname{RH}_{q_0^*}}^{\frac12-\frac{\gamma}{2}}   [w^{q_0^*}]_{A_{\infty}}^{1/2} \big(  [w^{q_0^*}]_{A_{\infty}} +\log[w]_{A_{\infty}}\big)^{1/2}\|f\|_{L^2(w)}\\
&\quad\lesssim  [w^{q_0^*}]_{A_{\phi(2)}}^{\frac{1}{q_0^*}}[w^{q_0^*}]_{A_{\phi(2)}}^{\frac{1}{q_0^*}(\frac{1}{2}-\frac{\gamma}{2})}[w^{q_0^*}]_{A_{\phi(2)}}^{\frac12}\Big([w^{q_0^*}]_{A_{\phi(2)}}+\frac{1}{q_0^*}\log([w^{q_0^*}]_{A_{(\phi(2)}})\Big)^{1/2}\|f\|_{L^2(w)}\\
&\quad= [w^{q_0^*}]_{A_{\phi(2)}}^{\frac{3-\gamma+q_0^*}{ 2q^*_0}}\Big([w^{q_0^*}]_{A_{\phi(2)}}+\frac{1}{q_0^*}\log([w^{q_0^*}]_{A_{(\phi(2)}})\Big)^{1/2}\|f\|_{L^2(w)}.
\end{align}
Now, Corollary \ref{corol:extrapol-weak} and Lemma \ref{lem:WeightProperties} (ii) yields, for all $p_{0} < q < q_{0}$,
\begin{align*}
&\|S(f)\|_{L^{q,\infty}(w)} \lesssim [w^{(q_0/q)'}]_{A_{\phi(q)}}^{\beta(2,q){\frac{3-\gamma+q_0^*}{ 2q^*_0}}}\Big([w^{(q_0/p)'}]_{A_{\phi(q)}}+\frac{1}{q_0^*}\log([w^{(q_0/q)'}]_{A_{\phi(q)}})\Big)^{\frac{\beta(2,q)}{2}}\|f\|_{L^q(w)}\\
&\quad \lesssim ([w]^{(q_0/q)'}_{A_{q/p_0}}[w]^{(q_0/q)'}_{\operatorname{RH}_{(q_0/q)'}})^{\beta(2,q){\frac{3-\gamma+q_0^*}{ 2q^*_0}}}\\
&\quad \quad\times\Big([w]^{(q_0/q)'}_{A_{q/p_0}}[w]^{(q_0/q)'}_{\operatorname{RH}_{(q_0/q)'}}+\frac{1}{q_0^*}\log([w]^{(q_0/q)'}_{A_{q/p_0}}[w]^{(q_0/q)'}_{\operatorname{RH}_{(q_0/q)'}})\Big)^{\frac{\beta(2,q)}{2}}\|f\|_{L^q(w)},
\end{align*}
which gives the desired bound.
\end{proof}

\begin{remark}
Notice that in \eqref{eq:loss} we have controlled our bound by a worse one to be able to apply the extrapolation theorem. This is precisely why we get weak type bounds which are worse than  the strong ones for certain $p$'s. 
\end{remark}


\end{document}